\documentclass[12pt]{amsart}
\usepackage{amssymb, amscd, amsmath, amsthm, epsfig, latexsym, enumerate}

\renewcommand{\geq}{\geqslant}
\renewcommand{\leq}{\leqslant}

\newtheorem{theorem}{Theorem}
\newtheorem{lemma}[theorem]{Lemma}
\newtheorem{cor}[theorem]{Corollary}

\newtheorem*{cor*}{Corollary}

\begin{document}

\title[elementary amenable groups]
{elementary amenable  groups of cohomological dimension 3}

\author{Jonathan A. Hillman }
\address{School of Mathematics and Statistics\\
     University of Sydney, NSW 2006\\
      Australia }

\email{jonathanhillman47@gmail.com}

\begin{abstract}
We show that torsion-free elementary amenable groups of Hirsch length $\leq3$ are solvable, of derived length $\leq3$.
This class includes all solvable groups of cohomological dimension $3$.
We show also that groups in the latter subclass are either polycyclic,
semidirect products $BS(1,n)\rtimes\mathbb{Z}$,
or properly ascending HNN extensions with base $\mathbb{Z}^2$ 
or $\pi_1(Kb)$.
\end{abstract}


\maketitle

D.Gildenhuys showed that the solvable groups of cohomological dimension at most 2 are either subgroups of the rationals $\mathbb{Q}$
or to be solvable Baumslag-Solitar groups $BS(1,m)$ \cite{Gi79}.
In particular, every such group has Hirsch length at most 2.
We show that finitely generated, torsion-free elementary amenable groups 
of Hirsch length $\leq3$ are in fact solvable minimax groups, 
of derived length $\leq3$.
We show also that such a group is finitely presentable if and only if it is constructible,
and such groups are either polycyclic,
semidirect products with base a solvable Baumslag-Solitar group,
or properly ascending HNN extensions with base $\mathbb{Z}^2$ 
or $\pi_1(Kb)$.
Our interest in this class of groups arose from recent work 
on aspherical 4-manifolds with non-empty boundary and elementary amenable fundamental group \cite{DH22}.
Such groups have cohomological dimension $\leq3$ and are of type $FP$,
and thus are in the class considered here.
(One of the results of \cite{DH22} is that the groups arising there are all 
either polycyclic or solvable Baumslag-Solitar groups, 
and so may be considered well understood.)

\section{background}

Let $G$ be a torsion-free elementary amenable group of finite Hirsch length 
$h=h(G)$.
Then $G$ is virtually solvable \cite{HL92},
and so has a subgroup of finite index which is an extension of a finitely generated free abelian group $\mathbb{Z}^v$ by a nilpotent group  \cite{Ca60}.
Since $v\leq{h}<\infty$ we may assume that $v$ is the virtual first Betti number 
of $G$,
i.e., the maximum of the ranks of abelian quotients of 
subgroups of finite index in $G$. 
If $G\not=1$ then $0<v\leq{h}=h(G)\leq{c.d.G}\leq{h+1}$.

We recall that the {\it Hirsch-Plotkin radical} $\sqrt{G}$ of a group $G$
is the (unique) maximal locally nilpotent normal subgroup of the group. 
(For the groups $G$ considered below, 
either $\sqrt{G}$ is abelian or $G$ is virtually nilpotent.)
If $G$ is solvable and $\sqrt{G}$ is abelian then $\sqrt{G}$ 
is its own centralizer in $G$ (by the maximality assumption),
and so the homomorphism from $G/\sqrt{G}$ to $Aut(\sqrt{G})$
induced by conjugation in $G$ is a monomorphism.

A solvable group is {\it minimax\/} if it has a composition series whose sections are either finite or isomorphic to $\mathbb{Z}[\frac1m]$, for some $m>0$.
A  solvable group is {\it constructible\/} if it is in the smallest class containing 
the trivial group which is closed under finite extensions and HNN extensions \cite{BB76}.
If $G$ is a torsion-free virtually solvable group group then
$c.d.G=h~\Leftrightarrow~G$ is of type $FP~\Leftrightarrow~G$ 
is constructible \cite{Kr86}.

Let $BS(m,n)$ be the Baumslag-Solitar group 
with presentation 
\[
\langle{a,t}\mid {ta^mt^{-1}=a^n}\rangle,
\]
and let $\overline{BS}(m,n)$ be the metabelian quotient
$BS(m,n)/\langle\langle{a}\rangle\rangle'$,
where $\langle\langle{a}\rangle\rangle'$ is
the commutator subgroup of the normal closure of the image of $a$
in $BS(m,n)$.
We may assume that $m>0$ and $|n|\geq{m}$.
(When $m=1$ and $n=\pm1$ we get $\mathbb{Z}^2$ and $\pi_1(Kb)$.)
Since we are only interested in torsion-free groups we shall assume also
that $(m,n)=1$.

\section{ hirsch length 2}

In this section we shall consider groups of Hirsch length 2,
which arise naturally in the analysis of groups of Hirsch length 3.
(Note also that some groups of Hirsch length 2
have cohomological dimension 3.)

\begin{theorem}
\label{h=2}
Let $G$ be a torsion-free elementary amenable group of Hirsch length $2$.
Then $\sqrt{G}$ is abelian,
and either $\sqrt{G}$ has rank $1$ and $G\cong\sqrt{G}\rtimes\mathbb{Z}$
or $\sqrt{G}$ has rank $2$ and $[G:\sqrt{G}]\leq2$.
\end{theorem}

\begin{proof}
Since $G$ is virtually solvable \cite{HL92} and the lowest non-trivial term of  
the derived series of a solvable group is 
a non-trivial abelian normal subgroup, $\sqrt{G}\not=1$.
Since any two members of $\sqrt{G}$ generate a torsion-free nilpotent group of Hirsch length $\leq2$ they commute.
Hence $\sqrt{G}$ is abelian, of rank $r=1$ or 2, say, and $h(G/\sqrt{G})=2-r$.

Let $C=C_G(\sqrt{G})$ be the centralizer of $\sqrt{G}$ in $G$.
If $N\leq{C}$ is a normal subgroup of $G$ with locally finite image in 
$G/\sqrt{G}$ then $N'$ is locally finite, 
by an easy extension of Schur's Theorem \cite[10.1.4]{Ro}.
Hence $N'=1$, so $N$ is abelian, and then $N\leq{\sqrt{G}}$,
by the maximality of $\sqrt{G}$.
Therefore any locally finite normal subgroup of $G/\sqrt{G}$ must
act effectively on $\sqrt{G}$.

If $\sqrt{G}$ has rank 1  then $G/\sqrt{G}$ can have no non-trivial torsion normal subgroup.
If $C\not=\sqrt{G}$ is infinite then it has an infinite abelian normal subgroup
(since it is non-trivial, virtually solvable, and has no non-trivial torsion normal subgroup).
But the preimage of any such subgroup in $G$ is nilpotent
(since it is a central extension of an abelian group).
This contradicts the maximality of $\sqrt{G}$.
Hence $=\sqrt{G}$ and so $G/\sqrt{G}$ acts effectively on $\sqrt{G}$.
Since $h(G/\sqrt{G})=1$ and $Aut(\sqrt{G})\leq\mathbb{Q}^\times$,
and $G/\sqrt{G}$ has no normal torsion subgroup,
we see that $G/\sqrt{G}\cong\mathbb{Z}$.

If $\sqrt{G}$ has rank 2 then $G/\sqrt{G}$ is a torsion group,
and $Aut(\sqrt{G})$ is isomorphic to a subgroup of $GL(2,\mathbb{Q})$.
If $G/\sqrt{G}$ is infinite then it must have an infinite locally finite normal subgroup
(since it is a virtually solvable torsion group).
But finite subgroups of $GL(2,\mathbb{Q})$ have order dividing 24,
and so $G/\sqrt{G}$ is finite. 
If $g$ in $G$ has image of finite order $p>1$ in $G/\sqrt{G}$
then conjugation by $g$ fixes $g^p\in\sqrt{G}$.
It follows that $g$ must have order 2 and its image in $GL(2,\mathbb{Q})$ 
must have determinant $-1$.
Hence $[G:\sqrt{G}]\leq2$.
\end{proof}

If $G$ is finitely generated then $\sqrt{G}$ is finitely generated as a module over
$\mathbb{Z}[G/\sqrt{G}]$, with respect to the action induced by conjugation in $G$.
If $h(\sqrt{G})=1$ then $\sqrt{G}$ is not finitely generated as an abelian group, 
while $G/\sqrt{G}\cong\mathbb{Z}$.
Hence  $\mathbb{Z}[G/\sqrt{G}]\cong\mathbb{Z}[t,t^{-1}]$,
and the action of $t$ is multiplication by some 
$\frac{n}m\in\mathbb{Q}\setminus\{0,\pm1\}$,
since $\sqrt{G}$ is torsion-free and of rank 1.
After replacing $t$ by $t^{-1}$, if necessary, 
we may assume that $\sqrt{G}\cong\mathbb{Z}[t,t^{-1}]/(mt-n)$, 
for some $m,n$ with $(m,n)=1$ and $|n|>m>0$.
Hence $G\cong\overline{BS}(m,n)$.
Then $c.d.G=2\Leftrightarrow~G$ is finitely presentable 
$\Leftrightarrow~{m}=1$ \cite{Gi79}.

If $G$ is finitely generated and $h(\sqrt{G})=2$ then $G\cong\mathbb{Z}^2$ or
$\pi_1(Kb)$, and so $c.d.G=2$.

Let $\mathbb{Z}_{(2)}$ be the localization of $\mathbb{Z}$ at $2$, 
in which all odd integers are invertible, and let $\mathbb{Z}_{(2)}$ act 
on $\mathbb{Q}$ through the surjection to 
$\mathbb{Z}_{(2)}/2\mathbb{Z}_{(2)}\cong\mathbb{Z}^\times=\{\pm1\}$.
Let  $\mathbb{Q}\otimes{Kb}$ be the extension of $\mathbb{Z}_{(2)}$ by 
$\mathbb{Q}$ with this action.
Then if $h=2$ and $G$ is not finitely generated it is either a subgroup of 
$\mathbb{Q}\rtimes_{\frac{m}n}\mathbb{Z}$,
for some nonzero $m,n$ with $(m,n)=1$  (if $h(\sqrt{G})=1$),
or is a subgroup of 
$\mathbb{Q}\otimes{Kb}$ (if $h(\sqrt{G})=2$).
Every such group has cohomological dimension 3.

\section{hirsch length 3}

Suppose now that $h(G)=3$. Then $h(\sqrt{G})=1$,  2 or 3.

\begin{theorem}
\label{h=3}
Let $G$ be a torsion-free elementary amenable group of Hirsch length $3$.
If $h(\sqrt{G})=1$ then $\sqrt{G}$ is abelian and $G/\sqrt{G}\cong\mathbb{Z}^2$.
If $h(\sqrt{G})=2$ then $\sqrt{G}$ is abelian and $G/\sqrt{G}\cong\mathbb{Z}$, 
$D_\infty$ or $\mathbb{Z}\oplus\mathbb{Z}/2\mathbb{Z}$.
If $h(\sqrt{G})=3$ then $G$ is virtually nilpotent.
In all cases, $G$ has derived length at most $3$.
\end{theorem}

\begin{proof}
If $h(\sqrt{G})=1$ then $\sqrt{G}$ is isomorphic to a subgroup of $\mathbb{Q}$ 
and (as in Theorem \ref{h=2}) $G/\sqrt{G}$ has no locally finite normal subgroup.
Since $C_G(\sqrt{G})$ is virtually solvable,
it follows that $C_G(\sqrt{G})=\sqrt{G}$ and so 
$G/\sqrt{G}$ embeds in $Aut(\sqrt{G})$, 
which is isomorphic to a subgroup of $\mathbb{Q}^\times$.
Hence $G/\sqrt{G}\cong\mathbb{Z}^2$,
and so $G$ has derived length 2.

If $h(\sqrt{G})=2$ then $\sqrt{G}$ is abelian and (as in Theorem \ref{h=2} again)
the maximal locally finite normal subgroup of $G/\sqrt{G}$ has order at most 2.
Since $G/\sqrt{G}$ is virtually solvable  and $h(G/\sqrt{G})=1$,
it has an abelian normal subgroup $A$ of rank 1,
which we may assume torsion-free and of finite index in $G/\sqrt{G}$.
Moreover, $G/\sqrt{G}$ embeds in $Aut(\sqrt{G})$, 
which is now isomorphic to a subgroup of $GL(2,\mathbb{Q})$.
No nontrivial element of $A$ can have both eigenvalues roots of unity,
for otherwise $C_G(\sqrt{G})>\sqrt{G}$.
Since the eigenvalues of $A$ have degree $\leq2$ over $\mathbb{Q}$,
it follows that no nontrivial element of $A$ can be infinitely divisible in $A$.
Hence $G/\sqrt{G}$ is virtually $\mathbb{Z}$, 
and so it is either $\mathbb{Z}$ or the infinite dihedral group 
$D_\infty=\mathbb{Z}/2\mathbb{Z}*\mathbb{Z}/2\mathbb{Z}$,
or an extension of one of these by $\mathbb{Z}/2\mathbb{Z}$.

If $G$ has a normal subgroup $H$ such that 
$H/\sqrt{G}\cong\mathbb{Z}/2\mathbb{Z}$ then conjugation in $G$ 
must preserve the filtration $0<H'<\sqrt{G}$ of $\sqrt{G}$.
Therefore elements of $G'$ act nilpotently on $\sqrt{G}$,
and so $G/H$ cannot be $D_\infty$.
Thus if $h(\sqrt{G})=2$ then $G/\sqrt{G}\cong\mathbb{Z}$, 
$D_\infty$ or $\mathbb{Z}\oplus\mathbb{Z}/2\mathbb{Z}$,
and $G$ has derived length 2, 3 or 2, respectively.

If $h(\sqrt{G})=3$ then $h(G/\sqrt{G})=0$, and so $G$ is virtually nilpotent.
Since iterated commutators live in finitely generated subgroups, 
the derived length of $G$ is the maximum of the derived lengths of its finitely generated subgroups.
Finitely generated torsion-free virtually nilpotent groups of Hirsch length 3
are polycyclic, and are fundamental groups of $\mathbb{N}il^3$-manifolds.
These are Seifert fibred over flat 2-orbifolds without reflector curves,
and so these groups have derived length $\leq3$.
Hence $G$ has derived length $\leq3$.
\end{proof}

\begin{cor}
\label{minimax}
If $G$ is finitely generated then it is a minimax group.
\end{cor}

\begin{proof}
If $h(\sqrt{G})=1$ and $G$ is finitely generated then $\sqrt{G}$ 
is finitely generated as a $\mathbb{Z}[\mathbb{Z}^2]$-module.
Since it is also torsion-free and of rank 1 as an abelian group, 
it is in fact a cyclic $\mathbb{Z}[\mathbb{Z}^2]$-module.
Hence $\sqrt{G}\cong\mathbb{Z}[\frac1D]$ for some $D>0$.

If $h(\sqrt{G})=2$ then $G$ has a subgroup $K$
of index $\leq2$ such that $K/\sqrt{G}\cong\mathbb{Z}$.
If $G$ is finitely generated then $K$ is also finitely generated.
Then $\sqrt{G}$ is again finitely generated as a $\Lambda$-module,
and is torsion-free and of rank 2 as an abelian group.
Hence it is isomorphic as a group to a subgroup of $\mathbb{Z}[\frac1m]^2$, 
for some $m>0$.

If $G$ is finitely generated and $h(\sqrt{G})=3$ then $G$ is polycyclic.
In all cases it is clear that $G$ is a minimax group.
\end{proof}

We shall consider more closely the cases with $h(\sqrt{G})=1$ or 2.

\begin{lemma}
If $G$ is finitely generated and $h(\sqrt{G})=1$ then
$G$ is a semidirect product $\overline{BS}(m,n)\rtimes\mathbb{Z}$,
where $mn$ has at least $2$ distinct prime factors.
\end{lemma}

\begin{proof}
If $h(\sqrt{G})=1$ then $G$ has a presentation  
\[
\langle{a,t,u}\mid {ta^mt^{-1}=a^n},~ua^pu^{-1}=a^q,~utu^{-1}=ta^e,
~\langle\langle{a}\rangle\rangle'~\rangle.
\]
for some nonzero $m,n,p,q$ with $(m,n)=(p,q)=1$ and some 
$e\in\mathbb{Z}[\frac1D]$,
where $D$ is the product of the prime factors of $mnpq$.
Hence $\sqrt{G}\cong\mathbb{Z}[\frac1D]$.
After a change of basis for $G/\sqrt{G}$, 
if necessary, 
we may assume that $mn$ has a prime factor which does not divide $pq$.
We may further arrange that $p$ divides $m$ and $q$ divides $n$, 
after replacing $t$ by $tu^N$ or $tu^{-N}$ for $N$ large enough, 
if necessary.
Hence $D$ is the product of the prime factors of $mn$.
It must have at least 2 prime factors,
since $G/\sqrt{G}\cong\mathbb{Z}^2$ maps injectively to
$Aut(\sqrt{G})\cong\mathbb{Z}[\frac1D]^\times$.

Thus $G\cong\overline{BS}(m,n)\rtimes_\theta\mathbb{Z}$,
for some automorphism $\theta$ of $\overline{BS}(m,n)$.
\end{proof}

\begin{theorem}
\label{coherence}
A finitely generated torsion-free elementary amenable group $G$ 
of Hirsch length $3$ is coherent if and only if it is $FP_2$ and $h(\sqrt{G})\geq2$.
\end{theorem}

\begin{proof}
If $G$ is coherent then it is finitely presentable and hence $FP_2$.

Suppose that $h(\sqrt{G})=1$. 
Then $\sqrt{G}\cong\mathbb{Z}[\frac1D]$ for some $D>1$,
and the image of $G/\sqrt{G}$ in 
$Aut(\sqrt{G})\cong\mathbb{Z}[\frac1D]^\times$ has rank 2.
Hence it contains a proper fraction $\frac{p}q$ with $p,q\not=\pm1$,
and so $G$ has a subgroup isomorphic to $\overline{BS}(p,q)$.
Since this subgroup is not even $FP_2$ \cite{BS78}, 
$G$ is not coherent.

If $h(\sqrt{G})=2$ then we may assume that
$G/\sqrt{G}\cong\mathbb{Z}$.
If, moreover,  $G$ is $FP_2$ then $G$ is an HNN extension with base 
a finitely generated subgroup of $\sqrt{G}$ \cite{BS78},
and the HNN extension is ascending, 
since $G$ is solvable.
Any finitely generated subgroup of $G$ is either a subgroup of the base or 
is itself an ascending HNN extension with finitely generated base, 
and so is finitely presentable.

If $h(\sqrt{G})=3$ then $G$ is polycyclic, 
and every subgroup is finitely presentable.
\end{proof}

It remains an open question whether an $FP_2$ torsion-free solvable group $G$ 
with $h(G)=3$ and $h(\sqrt{G})=1$ must be finitely presentable.
Note also that the argument shows that $G$ is {\it almost coherent}
(finitely generated subgroups are $FP_2$) if and only if it is coherent.

We shall assume next that $h(\sqrt{G})=2$ and that $G/\sqrt{G}\cong\mathbb{Z}$. Since $\mathbb{Q}\otimes\sqrt{G}\cong\mathbb{Q}^2$,
the action of $G/\sqrt{G}$ on $\sqrt{G}$ by conjugation in $G$ 
determines a conjugacy class of matrices $M$ in $GL(2,\mathbb{Q})$.
Hence $G\cong\sqrt{G}\rtimes_M\mathbb{Z}$. 

\begin{lemma}
\label{Zmatrix}
A matrix $M\in{GL(2,\mathbb{Q})}$ is conjugate to an integral matrix if and only if
$\det{M}$ and $tr\, M\in\mathbb{Z}$.
\end{lemma}

\begin{proof}
These conditions are clearly necessary.
If they hold then the characteristic polynomial is a monic polynomial 
with $\mathbb{Z}$ coefficents.
If $x\in\mathbb{Q}^2$ is not an eigenvector for $M$ then 
the subgroup generated by $x$ and $Mx$ is a lattice.
Since $M$ preserves this lattice, 
by the Cayley-Hamilton Theorem,
 it is conjugate to an integral matrix.
\end{proof}

If  $G$ is finitely generated then $\sqrt{G}$ is finitely generated 
as a $\mathbb{Z}[G/\sqrt{G}]$-module, 
since $\mathbb{Z}[G/\sqrt{G}]$ is noetherian.
It is finitely generated as an abelian group (and so $G$ is polycyclic) 
$\Leftrightarrow{M}$ is conjugate to a matrix in $GL(2,\mathbb{Z})
\Leftrightarrow\det{M}=\pm1$ and $tr\,M\in\mathbb{Z}$.

If $G$ is $FP_2$ then $G$ is an ascending HNN extension with base 
$\mathbb{Z}^2$ (as in Theorem \ref{coherence} above).
Hence $M$ (or $M^{-1}$) must be conjugate to an integral matrix,
and $G$ is finitely presentable.
On the other hand, 
if $G\cong\sqrt{G}\rtimes_M\mathbb{Z}$ and neither $M$ nor $M^{-1}$
is conjugate to an integral matrix then $G$ cannot be $FP_2$.

We conclude this section by giving some examples realizing 
the other possibilities for $G/\sqrt{G}$ allowed for by Theorem \ref{h=3}.
Torsion-free polycyclic groups $G$ with $h(\sqrt{G})=2$ 
are $\mathbb{S}ol^3$-manifold groups.
There are such groups with
$G/\sqrt{G}\cong\mathbb{Z}$, $D_\infty$ or 
$\mathbb{Z}\oplus\mathbb{Z}/2\mathbb{Z}$.
(The examples with $G/\sqrt{G}\cong{D_\infty}$ are fundamental groups 
of the unions of two twisted $I$-bundles over a torus along their boundaries.)

For instance, the group $G$ with presentation
\[
\langle{u,v,y}\mid{uyu^{-1}=y^{-1}},~vyv^{-1}=v^{-2}y^{-1},~v^2=u^2y\rangle
\]
is a generalized free product with amalgamation $A*_CB$ where 
$A=\langle{u,y}\rangle\cong{B}=\langle{v,u^2y}\rangle\cong\pi_1(Kb)$
and $C=\langle{u^2,y}\rangle\cong\mathbb{Z}^2$.
It is clear that $G/C\cong{D_\infty}$, and it is easy to check that $C=\sqrt{G}$.

If $G$ is the group with presentation
\[
\langle{t,x,y}\mid{tx=xt},~tyt^{-1}=y^n,~xyx^{-1}=y^{-1}\rangle
\]
then $\sqrt{G}$ is normally generated by $x^2$ and $y$, so 
$h(\sqrt{G})=2$ and 
$G/\sqrt{G}\cong
\mathbb{Z}\oplus\mathbb{Z}/2\mathbb{Z}$.

If $G/\sqrt{G}\cong{D_\infty}$ then $G$ is generated by $\sqrt{G}$ 
and two elements $u$, $v$ with squares in $\sqrt{G}$. 
The matrices in $GL(2,\mathbb{Q})$ corresponding to the actions 
of $u$ and $v$ have determinant $-1$.
Hence $t=uv$ corresponds to a matrix with determinant 1.
There are finitely generated examples of this type which are not polycyclic.
For instance,
let $F$ be the group with presentation
\[
\langle{u,v,x,y}\mid{u^2=x},~uyu^{-1}=y^{-1},~v^2=xy,~vy^3v^{-1}=x^2y^{-1}
\rangle,
\]
and let $K$ be the normal closure of  the image of $\{x,y\}$ in $F$.
Then $F/K\cong{D_\infty}$ and $K/K'\cong\mathbb{Z}[\frac13]^2$,
and $F/K'$ is torsion-free, solvable and $h(F/K')=3$.

However, 
if such a group $G$ is $FP_2$ then so is the subgroup generated by 
$\sqrt{G}$ and $t$.
Hence this subgroup is an ascending HNN extension 
with finitely generated base $H\leq\sqrt{G}$ \cite{BS78}.
Since $t$ maps $H\cong\mathbb{Z}^2$ into itself and has determinant 1 
it must be an automorphism of $H$, 
and so $G$ is polycyclic.

\section{finitely presentable implies constructible}

In this section we shall show that if  a torsion-free solvable group
$G$ of Hirsch length 3 is finitely presentable 
then it is in fact constructible,
and we shall describe all such groups.

If $G$ is $FP_2$ and $G/G'$ is infinite then $G$ is an HNN extension 
$H*_\varphi$ with finitely generated base $H$ \cite{BS78},
and the extension is ascending since $G$ is solvable.
Clearly $h(H)=h(G)-1=2$, and $c.d.G\leq{c.d.H+1}$.
In our next theorem we shall need the stronger hypothesis that $G$ be finitely presentable.

\begin{theorem}
\label{fpres-constructible}
Let $G$ be a torsion-free solvable group of Hirsch length $3$.
Then $G$ is finitely presentable if and only if it is constructible.
\end{theorem}

\begin{proof}
If $G$ is constructible then it is finitely presentable.
Assume that $G$ is finitely presentable.
If $\sqrt{G}$ has rank 1 then $G$ has a presentation
\[
\langle{a,t,u}\mid {ta^mt^{-1}=a^n},~ua^pu^{-1}=a^q,~utu^{-1}t^{-1}=C(a,t,u),~
R~
\rangle,
\]
for some nonzero $m,n,p,q$ with $(m,n)=(p,q)=1$ 
and word $C(a,t,u)$ of weight 0 in each of $t$ and $u$,
and some finite set of relators $R$.
Let $D$ be the product of the prime factors of $mnpq$.
Then $\sqrt{G}\cong\mathbb{Z}[\frac1D]$,
and contains the image of $c$ in $G$.
As observed after Corollary \ref{minimax}, we may assume that
$p$ and $q$ divide $m$ and $n$, respectively
and that $mn$ has a prime factor which does not divide $pq$.

We may assume that each of the relations in $R$ has weight 0 in each of $t$ and $u$.
Then we may write $C(a,t,u)$ and each relator in $R$ as 
a product of conjugates $b_{i,j}=t^iu^jau^{-j}t^{-i}$ of $a$.
Since $R$ is finite the exponents $i,j$ involved lie in a finite range 
$[-L,L]$, for some $L\geq0$.
The relations imply that the normal closure of the image of 
$a$ in $G$ is $\sqrt{G}\cong\mathbb{Z}[\frac1D]$.
Hence the images of the $b_{i,j}$s in $G$ commute,
and are powers of an element $\alpha$ represented by a word $w=W(a,t,u)$
which is a product of powers of (some of) the $b_{i,j}$s.
In particular, $a=\alpha^N$ and  $b_{i,j}=\alpha^{e(i,j)}$,
for some exponents $N$ and $e(i,j)$. 
Clearly $N=e(0,0)$.

It follows also that $t\alpha^mt^{-1}=\alpha^n$ and $u\alpha^pu^{-1}=\alpha^q$.
Hence adjoining a new generator $\alpha$ and new relations 
\begin{enumerate}
\item$a=\alpha^N$; 
\item$t\alpha^mt^{-1}=\alpha^n$;
\item$u\alpha^pu^{-1}=\alpha^q$; 
\item$\alpha=W(a,t,u)$;
and 
\item
$t^iu^jau^{-j}t^{-i}=\alpha^{e(i,j)}$, for all $i,j\in[-L,L]$.
\end{enumerate}
gives an equivalent presentation.

We may use the first relation to eliminate the generator $a$.
Since the image of $\alpha$ in $G$ generates an infinite cyclic subgroup,
the relations $R$ must be consequences of these,
and so we may delete the relations in $R$.
Moreover the relation $\alpha=W(a,t,u)$ collapses to a tautology,
and so may also be deleted, 
and we may use the final set of relations to write $C(a,t,u)$ as a power of $\alpha$.
Since $tb_{i,j}t^{-1}=b_{i+1,j}$ and $ub_{i,j}u^{-1}=b_{i,j+1}$,
we see that $e(i,j)=(\frac{n}m)^i(\frac{q}p)^je(0,0)$, for all $i,j\in[-L,L]$.
Since $\alpha$ generates the subgroup spanned by the $b_{i,j}$s it follows that
$N=(mnpq)^L$ and $e(i,j)=N(\frac{n}m)^i(\frac{q}p)^j$ for $i,j\in[-L,L]$.
Hence the final set of relations are consequences of the second and third relations.

Thus $G$ has the finite presentation
\[
\langle{t,u,\alpha}\mid {t\alpha^mt^{-1}=\alpha^n},~u\alpha^pu^{-1}=\alpha^q,
~utu^{-1}t^{-1}=\alpha^c\rangle,
\]
for some $c\in\mathbb{Z}$.
Since the subgroup generated by the images of $t$ and $\alpha$ is
isomorphic to $BS(m,n)$ and is solvable, either $m$ or $n=1$ \cite{BS78}.

If $h(\sqrt{G})=2$ then $G$ has a subgroup $J$ of index $\leq2$
which is an ascending HNN extension with finitely generated base 
$H\leq{\sqrt{G}}$.
Since $h(H)=2$, we have $H\cong\mathbb{Z}^2$.
Hence  $J$ is constructible, and $G$ is also constructible.

If $h(\sqrt{G})=3$ then $G$ is virtually nilpotent,
and so is again constructible.
\end{proof}

\begin{theorem}
\label{list constructible}
Let $G$ be a torsion-free elementary amenable group of Hirsch length $3$.
Then $G$ is constructible if and only if either
\begin{enumerate}
\item$G\cong{BS(1,n)\rtimes_\theta\mathbb{Z}}$ for some 
$n\not=0$ or $\pm1$  and some $\theta\in{Aut(BS(1,n))}$;
\item$G\cong{H}*_\varphi$ is a properly ascending HNN extension with base 
$H\cong\mathbb{Z}^2$ or $\pi_1(Kb)$; or
\item$G$ is polycyclic.
\end{enumerate}
\end{theorem}

\begin{proof}
It shall suffice to show that if $G$ is constructible then it is one of the groups
listed here, as they are all clearly constructible.
We may also assume that $G$ is not polycyclic, 
and so $h(\sqrt{G})=1$ or 2.

Since $G$ is constructible it has a subgroup $J$
of finite index which is an ascending HNN extension with base 
a constructible solvable group of Hirsch length 2. 
Since $G$ is not polycyclic, we may assume that $J=G$, by Theorem \ref{h=3}
(when $h(\sqrt{G})=1$) and by Theorem \ref{h=3} with the observations 
towards the end of \S3 (when $h(\sqrt{G})=2$).
Constructible solvable groups of Hirsch length 2 are in turn Baumslag-Solitar 
groups $BS(1,m)$ with $m\not=0$.

If $h(\sqrt{G})=1$ then $|m|>1$ and $G\cong{BS(1,m)}*_\varphi$, 
for some injective endomorphism of $BS(1,m)$. 
We shall use the presentation for $BS(1,m)$ given in \S2. 
After replacing $a$ by $t^{-k}at^k$, if necessary,
we may assume that $\varphi(a)=a^q$ and $\varphi(t)=ta^r$, 
for some $q\not=0$ and $r$ in $\mathbb{Z}$.
Then $G$ has a presentation
\[
\langle{a,t,u}\mid {tat^{-1}=a^m},~uau^{-1}=a^q,~utu^{-1}=ta^r\rangle.
\]
Let $s=tu$ and $n=mq$. Then $sas^{-1}=a^n$, and 
the subgroup $H\cong{BS(1,n)}$ generated by $a$ and $s$ is normal in $G$.
Conjugation by $u$ generates an automorphism $\theta$ of $H$,
since $q$ is invertible in $\mathbb{Z}[\frac1n]$.
Hence $G\cong{BS(1,n)}\rtimes_\theta\mathbb{Z}$, and so $G$ is of type (1).

If $h(\sqrt{G})=2$ then $m=\pm1$,  and so 
$H\cong\mathbb{Z}^2$ or $\pi_1(Kb)$.
Since the HNN extension is properly ascending,
$G$ is not polycyclic, and so $G$ is of type (2).
\end{proof}

We have allowed an overlap between classes (1) and (2) 
in Theorem \ref{list constructible}, 
for simplicity of formulation.
Polycyclic groups of Hirsch length 3 
are virtually semidirect products $\mathbb{Z}^2\rtimes\mathbb{Z}$.
Such semidirect products are ascending HNN extensions, 
but the extensions are not properly ascending, 
and so classes (2) and (3) are disjoint.

Taking into account the fact that solvable groups $G$
with $c.d.G=h(G)$ are constructible \cite{Kr86},
we may summarize the above two theorems as follows.

\begin{cor}
If $G$ is a torsion-free elementary amenable group of
Hirsch length $3$ then $c.d.G=3~\Leftrightarrow~G$ is constructible
$\Leftrightarrow~G$ is finitely presentable $\Leftrightarrow~G$ is one 
of the groups listed in Theorem \ref{list constructible} above.
\qed
\end{cor}

We conclude with some remarks on realizing such groups 
as fundamental groups of aspherical manifolds.
If $G$ is a finitely presentable group of type $FF$ then there is
a finite $K(G,1)$-complex of dimension $\max\{3,c.d.G\}$
\cite{Wall}.
Thickening such a complex gives a compact aspherical manifold 
of twice the dimension and with fundamental group $G$.
We may define the {\it manifold dimension\/} of $G$ to be 
the miminal dimension $m.d.G$ of such a manifold.
If $c.d.G=h$ (and $h\not=2$) then there is a finite $K(G,1)$-complex of
dimension $h$, and so $m.d.G\leq2h$.
If $G$ is virtually polycyclic then $K(G,1)$ is homotopy equivalent 
to a closed $h$-manifold, and so $m.d.G=h$.
However if $G$ is not virtually polycyclic then $m.d.G>h+1$
\cite{DH22}.

In particular, 
groups of the first two types allowed by Theorem \ref{list constructible}
are not realizable by aspherical 4-manifolds.
The 2-complex associated to the standard 1-relator presentation of $BS(1,m)$ is aspherical, and so $m.d.BS(1,m)\leq4$ (with equality if $m\not=\pm1$).
Surgery arguments show that every automorphism $\theta$ of $BS(1,m)$ 
is induced by a self-homeomorphism $\Theta$ of such a 4-manifold \cite{DH22}.
The mapping torus of $\Theta$ is an aspherical 5-manifold,
and so $m.d.{BS(1,m)}\rtimes_\theta\mathbb{Z}=5$ (if $m\not=\pm1$).
The question remains open for properly ascending HNN extensions with base
$\mathbb{Z}^2$ or $\pi_1(Kb)$.

%\newpage

\end{document}